\documentclass[a4paper,12pt]{amsart} 

\usepackage{comandi}

\begin{document}

\title[$\Lambda$-modules and holomorphic Lie algebroid connections]{$\Lambda$-modules and holomorphic Lie algebroid connections}

\author{Pietro Tortella}

\address{Mathematical Physics Sector, SISSA, Via Bonomea 265, 34136, Trieste (ITALY)\newline
Laboratoire Paul Painlev\'e, Universit\'e Lille 1, Cit\'e Scientifique, 59655, Villeneuve D'Ascq Cedex (FRANCE)}

\email{pietro.tortella@sissa.it}

\begin{abstract}
Let $X$ be a complex smooth projective variety, and $\belG$ a locally free sheaf on $X$. We show that there is a 1-to-1 correspondence between pairs $(\Lambda,\Xi)$, where $\Lambda$ is a sheaf of almost polynomial filtered algebras over $X$ satisfying Simpson's axioms and $\Xi: \Gr\Lambda \rightarrow \Sym^\bullet_{\corO_X} \belG$ is an isomorphism, and pairs $(\bella,\Sigma)$, where $\bella$ is a holomorphic Lie algebroid structure on $\belG$ and $\Sigma$ is a class in $F^1H^2(\bella,\C)$, the first Hodge filtration piece of the second cohomology of $\bella$.

As an application, we construct moduli spaces of semistable flat $\bella$-connections for any holomorphic Lie algebroid $\bella$. Particular examples of these are given by generalized holomorphic bundles for any generalized complex structure associated to a holomorphic Poisson manifold.
\end{abstract}

\maketitle

\section{Introduction}
In \cite{simpson_representation_1}, C. Simpson constructed, by using GIT techniques, moduli spaces for semistable $\Lambda$-modules over a smooth projective algebraic variety $X$, where $\Lambda$ is a sheaf of filtered $\C_X$-algebras satisfying some axioms. With this result, by varying the algebra $\Lambda$, one gets moduli spaces of a large class of objects, which includes semistable coherent sheaves, flat connections and Higgs bundles.

The main results of this paper are classification of the sheaves of algebras $\Lambda$ satisfying Simpson's axioms, and some applications.

In \cite{sridharan}, Sridharan classifies filtered $k$-algebras $A$ such that $A_{(0)} = k$ and the associated graded algebra is the symmetric algebra over the first graded piece $\Gr_1 A$. He shows the following:
\begin{theorem}
Let $k$ be a commutative unital ring, $L$ a free $k$-module and $\belS = \Sym_k^\bullet L$ the full symmetric algebra of $L$.

Then there is a 1-to-1 correspondence between:
\begin{itemize}
\item isomorphism classes of pairs $(A,\Xi)$, where $A$ is a filtered $k$-algebra with $A_{(0)} = k$, $\Gr_1 A \iso L$ and $\Xi : \Gr A \rightarrow \belS$ an isomorphism of graded algebras;
\item pairs $(\gotg,\Sigma)$, where $\gotg$ is a $k$-Lie algebra structure on $L$ and $\sigma \in H^2_{CE}(\gotg,k)$, the second Chevalley-Eilenberg cohomology group.
\end{itemize}
\end{theorem}

We will generalize Sridharan's construction to the case of sheaves of filtered algebras over a smooth projective variety satisfying Simpson's axioms: the first step is to drop the hypothesis $A_{(0)}=k$, and the second one is a generalization to sheaves. What we will obtain is the following:
\begin{theorem} \label{thm:main_intro}
Let $X$ be a compact K\"ahler manifold, $\belG$ a locally free coherent $\corO_X$-module and $\belS = \Sym_{\corO_X}^\bullet \belG$.

Then there is a 1-to-1 correspondence between
\begin{itemize}
\item isomorphism classes of pairs $(\Lambda,\Xi)$ where $\Lambda$ is a sheaf of filtered algebras satisfying Simpson's axioms and $\Xi : \Gr\Lambda \rightarrow \belS$ is an isomorphism of graded algebras;
\item pairs $(\bella,\Sigma)$ where $\bella$ is a holomorphic Lie algebroid structure on $\belG$ and $\Sigma \in F^1H^2(\bella,\C)$.
\end{itemize}
\end{theorem}
The vector space $F^1H^2(\bella,\C)$ is the first piece of the Hodge filtration of the $2$nd Lie algebroid cohomology group $H^2(\bella,\C)$, and is computed using the cohomological theory for holomorphic Lie algebroids developed in \cite{xu}.

We then present some applications of this correspondence in the study of moduli spaces of $\Lambda$-modules. One can observe that the data of a triple $(H, \delta , \gamma)$ as in the second section of \cite{simpson_representation_1} is equivalent to the data of a holomorphic Lie algebroid, and that our theorem completes Theorem 2.11 of loc. cit. As a consequence, we obtain moduli spaces of flat $\bella$-connections for any holomorphic Lie algebroid $\bella$, and with the functoriality property we construct rational maps between these spaces.

Moreover, we expect that through this correspondence one can apply the theory of Lie algebroids in the study of moduli spaces of $\Lambda$-modules. In a paper to come, we will apply the deformation theory of Lie algebroids to construct a generalization of Deligne-Simpson $\lambda$-connections \cite{simpson_lambda} and to compactify the moduli spaces of flat connections by allowing the latter to degenerate along foliations.

\spazio

The paper is organized as follows: Section 2 is standard, and it aims to give a quick introduction to the objects we are using and fix the notation. We first define (smooth) Lie algebroids and their cohomology. Then we introduce the $\la$-characteristic ring of a vector bundle. Finally we give a definition of a matched (or twilled) pair of Lie algebroids. An exhaustive discussion on these subjects can be found in \cite{mckenzie}, \cite{fernandes}, \cite{huebschmann_twilled}.

In Section 3 we discuss holomorphic Lie algebroids. It contains mainly known facts, but we add some small improvements that we have not found in the literature. After standard definitions, we recall from \cite{xu} the construction of the canonical complex Lie algebroid $\bella_h$ associated to a holomorphic Lie algebroid $\bella$ and a generalization of the holomorphic deRham Theorem, while from \cite{ugo_volodya} we take a generalization of Dolbeault's theory. We then prove a kind of partial degeneration of the spectral sequence computing the cohomology of a holomorphic Lie algebroid, that will allow us to describe the vector spaces $F^1H^2(\bella,\C)$. 
Further, we adress the question of existence of holomorphic $\bella$-connections on coherent $\corO_X$-modules: as in \cite{ref_1} and \cite{ref_2}, we introduce the jet bundle of a Lie algebroid and as in \cite{calaque} the $\bella$-Atiyah class of a coherent $\corO_X$-module $\belE$; this generates the $\bella_h$-characteristic ring of $\belE$, and we show that, also when $\belE$ is torsion free, the trace of the curvature of a holomorphic $\bella$-connection is a representative of the first $\bella_h$-Chern class of $\belE$.

Section 4 is devoted to the proof of Theorem \ref{thm:main_intro}: we first classify the extensions of a holomorphic Lie algebroid $\bella$ by $\corO_X$ on a compact K\"ahler manifold, and then generalize Sridharan's construction to our context. Remark that in \cite{BB} the same kind of classification is provided in the case when $\bella$ is the canonical Lie algebroid $\belT_X$.

In the last section, we show an application of this classification to moduli spaces of $\Lambda$-modules: it is well known that if $\bella$ is a Lie algebroid, then a $\bella$-representation is equivalent to a $\belU (\bella)$-module, where $\belU(\bella)$ is the universal enveloping algebra of $\bella$. We will see that module structures for algebras $\Lambda$ twisted by a cohomology class $\Sigma$ can be described as bunches of local $\bella$-connections on $\belE$ satisfying some compatibility equations prescribed by $\Sigma$. These have a better interpretation as representations of torsors, as we will describe in a forthcoming paper. 
In particular, we will see that by using Simpson's techniques one obtains a quasi-projective moduli scheme for flat $\bella$-connections for any holomorphic Lie algebroid $\bella$,
while the results of Section 3 allow us to conclude that the moduli spaces corresponding to a non-zero $\Sigma \in F^2H^2(\bella;\C)$ are empty.
Examples of particular interest are then examined, and we will see, following \cite{xu}, that when $\bella$ is the holomorphic Lie algebroid $(\Omega_X)_{\Pi}$ associated to a holomorphic Poisson manifold $(X,\Pi)$, $\Lambda$-modules coincide with generalized holomorphic bundles, so that Simpson's construction provides moduli spaces for these objects.

\subsection*{Acknowledgement}
This work covers a part of my PhD thesis. I am very indebted to my advisors Ugo Bruzzo and Dimitri Markushevich for proposing the problem and their constant guidance through my work. Moreover I would like to thank Vladimir Rubtsov, Carlos Simpson, Rui Loja Fernandes  and David Mart\'inez Torres for very useful conversations. 

I acknowledge a partial support by PRIN 2008: "Geometria delle variet\`a algebriche e dei loro spazi di moduli", the grant ANR-09-BLAN-0104-01, and GRIFGA.

\section{Lie algebroids}
\subsection{Real Lie algebroids} \label{sec:real_lie}

Let $X$ be a smooth manifold and $T_X$ its tangent bundle. 

\begin{definition}
A (real) \emph{Lie algebroid} on $M$ is a triple $(\la,a,\{\cdot,\cdot\})$ such that:
\begin{itemize}
\item $\la$ is a vector bundle on $M$;
\item $\{\cdot,\cdot\}$ is a $\R$-Lie algebra structure on $\Gamma(\la)$, the space of smooth global sections 
of $\la$;
\item $a: \la\rightarrow T_M$ is a vector bundle morphism (the anchor) that induces a Lie algebra morphism 
on global sections and satisfies the following Leibniz rule: 
$$\{u,fv\} = f\{u,v\} + a(u)(f) v$$
 for any $u,v \in \Gamma(\la)$ and $f \in C^\infty(M)$.
\end{itemize}
\end{definition}

A morphism between two Lie algebroids $\la$ and $\la'$ over the same base manifold $M$ is a vector bundle 
morphism $\la\rightarrow \la'$ that commutes with the anchors and such that the induced morphism on global sections 
is a Lie algebra morphism.

For $p>1$ we denote by $\belA^p_L$ the sheaf of smooth sections of $\bigwedge^p L^*$ and by $A^p_L$ the 
$C^\infty(X)$-module of its global sections. Set $\belA^0_L = \belA^0_X$, the sheaf of smooth
functions on $X$. Elements of $\belA^k_L$ ($A^k_L$) are called (global) $k$-$L$-forms.

One can define a $L$-differential $\dc_L : \belA^k_L \rightarrow \belA^{k+1}_L$ via the formula
\begin{equation} \label{eqn:differential_la}
\begin{array}{rcl}
(\dc_\la \theta) (u_1,\ldots,u_{p+1}) &=& \sum_i (-1)^{i+1} a(u_i) (\theta (u_1,\ldots,\hat{u}_i,\ldots,u_{p+1}))\ +\\
 &+& \sum_{i<j} (-1)^{i+j} \theta(\{u_i,u_j\}, u_1,\ldots, \hat{u}_i,\hat{u}_j, \ldots u_{p+1})
\end{array}\end{equation}
for $\theta \in \belA^p_\la$ and $u_1,\ldots,u_{p+1}$ (local) sections of $\la$.

One has $\dc_\la^2 = 0$, and we call the cohomology of the complex $(\bigwedge^\bullet \la\dual, \dc_\la)$ the 
\emph{Lie algebroid cohomology} of $\la$, and denote it by $H^\bullet(\la,a,\{\cdot,\cdot\},\R)$ 
(or simply $H^\bullet(\la,\R)$).

Remark that $a\dual: T^*_X \rightarrow \la\dual$, the dual of the anchor $a$, induces a morphism of complexes 
$\Gamma(\bigwedge^p T^*_X) \rightarrow \Gamma(\bigwedge^p \la\dual)$ yielding a morphism in cohomology 
$a^* : H_{DR}^p(X,\R) \rightarrow H^p(\la,\R)$.

\subsection{Basic examples} \label{sec:example_1}
The tangent bundle $T_X$ has a \emph{canonical} Lie algebroid structure with the identity as the anchor and the 
commutator of vector fields as the bracket. 

\spazio
Any regular integrable foliation $\corF\subseteq T_X$ carries a natural Lie algebroid structure, for wich the 
anchor is the inclusion and the bracket is the restriction to $\Gamma(\corF)$ of the commutator of vector fields.

\spazio
Let $(K,0, \{\cdot,\cdot\})$ be a Lie algebroid with anchor equal to zero. Then the bracket is actually 
$C^\infty(X)$-linear, so $\Gamma(K)$ inherits a $C^\infty(X)$-Lie algebra structure. Moreover, the bracket 
induces on each fiber of $K$ a $\R$-Lie algebra structure that varies smoothly. So $K$ is a so called 
\emph{bundle of Lie algebras}.

Vice versa, if $K$ is a bundle of Lie algebras, the brackets on the fibers glue to define a bracket on global 
sections, so we can define a Lie algebroid structure on $K$ using this bracket and the $0$ map as anchor.

Remark that we can see any vector bundle $K$ as a Lie algebra bundle with trivial Lie algebra structure, 
i.e. any vector bundle admits a trivial Lie algebroid structure, having both anchor and bracket equal to $0$.

\spazio
Let $(\la,a,\{\cdot,\cdot\})$ be a Lie algebroid: the image of the anchor $\corF = \image(a)$ is an integrable 
foliation of $X$ (not necessarily regular), that we call \emph{foliation associated to the Lie algebroid}.

The kernel of $a$ gives a subsheaf of the sections of $L$, and is naturally equipped with a structure
of sheaves of $\belA^0_X$-Lie algebra.

\spazio
Let $X$ be a smooth manifold and $\Pi\in \Gamma(\bigwedge^2T_X)$ a Poisson bivector. 
Then we can define a Lie algebroid structure on the cotangent bundle $T\dual_X$ with
\begin{itemize} 
\item anchor $\sharp : T\dual_X \rightarrow T_X$ given by the contraction with $\Pi$,
\item bracket defined by the formula
$$
\{\alpha,\beta\} = \dc \langle\Pi,\alpha\wedge\beta \rangle - \corL_{\sharp (\beta)} \alpha + 
\corL_{\sharp(\alpha)}\beta
$$
for any $\alpha,\beta\in \Gamma(T\dual_X)$, where $\corL_V$ is the Lie derivative along the vector field $V$.
\end{itemize}

\subsection{Lie algebroid connections and characteristic classes} 

As before, let $X$ be a manifold, $(\la,a,\{\cdot,\cdot\})$ be a Lie algebroid and $E$ a vector bundle over it. 

\begin{definition}
By a \emph{$\la$-connection} on $E$ we mean a map $\nabla: \Gamma(E) \rightarrow \Gamma(E)\otimes A_\la^1$ satisfying 
the Leibniz rule $\nabla(fe) = f\nabla e + e \otimes \dc_\la f$. 
\end{definition}

As with usual connections, one can extend an $\la$-connection to higher degree forms: 
let $\belA_\la^k(E)$ be the sheaf of sections of $\bigwedge^k \la\dual \otimes E$, 
and $A^k_\la(E)$ its global sections. Then there is a unique way to extend $\nabla$ to an operator
$\nabla: A_\la^k(E) \rightarrow A_\la^{k+1}(E)$, such that 
$\nabla(\eta \otimes s) = \dc_\la \eta \otimes s + (-1)^k \eta \wedge \nabla s$ 
for $\eta \in A_\la^k$ and $s\in \Gamma(E)$. 

Define the \emph{curvature} of $\nabla$ as 
$$
F_\nabla = \nabla\circ \nabla: A_\la^0(E) \rightarrow A_\la^2(E).
$$ 
It turns out that $F_\nabla$ is $C^\infty(X)$-linear, thus yielding a global section in $A^2_\la(\End E)$; 
moreover the following formula holds:
$$
F_\nabla (u,v)(e) = [\nabla_u,\nabla_v](e) -\nabla_{\{u,v\}}e\ , \qquad u,v\in \Gamma(\la),\ e\in \Gamma(E)\ ,
$$
where $\nabla_w:E\rightarrow E$ denotes the $1$st order differential operator 
$e\rightarrow \langle\nabla e, w\rangle$.

Remark that any ($T_X$-)connection on $E$ induces a $\la$-connection: 
if $\nabla: E \rightarrow E\otimes T^*_X$ is a connection, composing with 
$\id_E \otimes a^* : E\otimes T^*_X \rightarrow E\otimes \la\dual$ we obtain a $\la$-connection. 
So on any vector bundle $E$ there always exist $\la$-connections.

We say that a $\la$-connection on the vector bundle $E$ is \emph{flat} when its curvature vanishes; 
when this happens, the map $\Gamma(\la) \rightarrow \Der (E)$ given by $w \rightarrow \nabla_w$ is a morphism of 
$\R$-Lie algebras. We will also use the terms \emph{representation of $\la$} or \emph{$\la$-module} to mean a 
flat $\la$-connection on a vector bundle $E$. 

If $(E,\nabla)$ is a $\la$-module, then $(A_\la^\bullet(E),\nabla)$ forms a complex, and we can define the 
cohomology groups of $\la$ with values in $E$, and denote them by $H^k(\la;E,\nabla)$.

\spazio
Let $E$ be a rank $r$ complex vector bundle on $M$. Recall that the \emph{characteristic ring} of $E$,
that we denote by $\belR(E)$, is the image of the Chern-Weil homomorphism
$$
\lambda_E : I^\bullet(\GL(r,\C)) \rightarrow H^\bullet_{DR}(X; \C),
$$
where $I^\bullet(\GL(r,\C))$ is the algebra of $\Ad(\GL(r,\C))$-invariant multilinear maps
$P:\gl_r \times \ldots \times \gl_r \rightarrow \C$, and $\lambda_E$ is defined using the curvature of any 
connection on $E$ (see for example \cite{G-H}, Chapter 3.3).

If $(\la,a,\{\cdot,\cdot\})$ is a Lie algebroid, we have:
\begin{definition}
The \emph{$\la$-characteristic ring} of a complex vector bundle $E$ is the pull-back through the anchor of 
the characteristic ring of $E$:
$$
\belR_\la^\bullet(E) = a^*(\belR^\bullet(E)) \subseteq H^\bullet(\la,\C):= H^\bullet(\la;\R)\otimes \C\ .
$$
\end{definition}

It turns out that one can compute the $\la$-characteristic ring of $E$ via the curvature of any 
$\la$-connection on it, similarly to what happens in the usual case: if $F\in A^2_\la(\End E)$ is the 
curvature of a $\la$-connection and $P\in I^k(\GL_r)$, then one can show that the formula 
$$
\sum_{\sigma \in S_{2k}} (-1)^\sigma P(F(u_{\sigma(1)},u_{\sigma(2)}),\ldots, F(u_{\sigma(2k-1)},u_{\sigma(2k)}))
$$ 
for $u_1,\ldots,u_{2k} \in \Gamma(\la)$ is well defined and gives a closed $2$-$\la$-form, and that its 
cohomology class does not depend on the $\la$-connection chosen. This construction yields a homomorphism 
$\lambda^\la_E: I^\bullet(\GL_r) \rightarrow H^\bullet(\la)$ that makes the diagram
$$
\xymatrix{ I^\bullet(\GL(r,\C)) \ar[r] \ar[dr] & H^\bullet_{DR}(X,\C) \ar[d] \\
& H^\bullet(\la,\C)
}
$$
commutative.

\subsection{Matched (twilled) pairs of Lie algebroids}

Let $(\la_i,a_i, \{\cdot,\cdot\}_i)$ for $i=1,2$ be two Lie algebroids over the same manifold $X$.
We say that they form a \emph{matched pair} (or \emph{twilled pair}) if 
we are given a $\la_1$-module structure on $\la_2$ and a $\la_2$-module structure on 
$\la_1$ (that we denote both by $\nabla$) satisfying the following equations:
\begin{equation} \label{eqn:matched_pair}
\begin{array}{c}
[a_1(u_1),a_2(u_2)] = -a_1(\nabla_{u_2} u_1) + a_2(\nabla_{u_1}u_2)\ ,\\
\nabla_{u_1}(\{u_2,v_2\}_2) = \{\nabla_{u_1}u_2 , v_2\}_2 + \{u_2 , \nabla_{u_1} v_2\}_2 + 
\nabla_{\nabla_{v_2}u_1}u_2 - \nabla_{\nabla_{u_2}u_1}v_2 \ ,\\
\nabla_{u_2}(\{u_1,v_1\}_1) = \{\nabla_{u_2}u_1 , v_1\}_1 + \{u_1 , \nabla_{u_2} v_1\}_1 + 
\nabla_{\nabla_{v_1}u_2}u_1 - \nabla_{\nabla_{u_1}u_2}v_1 \ .
\end{array}\end{equation}

This definition is motivated by the following (see \cite{huebschmann_twilled}):
\begin{proposition}
\begin{enumerate}
\item
Let $\la$ be a Lie algebroid, and $\la_1,\la_2$ two sub-Lie algebroids of $\la$ such that
$\la = \la_1\oplus \la_2$. Then $(\la_1,\la_2)$ is naturally a matched pair of Lie algebroids.

\item
Let $(\la_1,\la_2)$ be a matched pair of Lie algebroids. Define on $\la = \la_1 \oplus \la_2$
the following structures: an anchor
$$
a: \la \rightarrow T_X \ , \qquad a(u_1+u_2)= a_1(u_1) + a_2(u_2)
$$ 
and a bracket
$$\begin{array}{rl}
\{u_1+u_2, v_1+v_2\} = &(\{u_1,v_1\}_1 + \nabla_{v_1}(u_2) - \nabla_{v_2}(u_1) ) + \\
  & +  (\{u_2,v_2\}_2+\nabla_{u_1}(v_2) - \nabla_{u_2}(v_1))
\end{array}$$
for any $u_i,v_i \in \Gamma(\la_i)$, $i=1,2$.

Then $(\la,a,\{\cdot,\cdot\})$ is a Lie algebroid, that we denote by $\la_1 \bowtie \la_2$,
and $\la_1,\la_2$ are naturally sub-Lie algebroids of $\la$.
\end{enumerate}
\end{proposition}

Consider the groups 
$$K^{p,q} = \Gamma(\bigwedge^p L_1\dual \otimes \bigwedge^q L_2\dual).$$ 
The $\la_i$-module structures on $\la_{\hat{i}}$ 
(where $\hat{i}= 2$ for $i=1$ and $\hat{i}=1$ for $i=2$) induces an $\la_i$-module 
structure on $\bigwedge^p \la\dual_{\hat{i}}$. This leads to two differentials 
$\dc_{\la_1}:K^{p,q} \rightarrow K^{p+1,q}$ and $\dc_{\la_2}: K^{p,q} \rightarrow K^{p,q+1}$ induced by the 
module structure. The three equations \ref{eqn:matched_pair} are equivalent to the commutation rule 
$\dc_{\la_1}\dc_{\la_2} = (-1)^p\dc_{\la_2}\dc_{\la_1}$, i.e. the triple 
$(K^{\bullet,\bullet},\dc_{\la_1},\dc_{\la_2})$ is a double complex.

One has:
\begin{proposition} \label{prop:cohom_matched_pair}
The Lie algebroid cohomology of $\la = \la_1\bowtie \la_2$ is the cohomology of the total complex
associated to $(K^{\bullet,\bullet} , \dc_{\la_1} , \dc_{\la_2})$.
\end{proposition}

\section{Holomorphic Lie algebroids} 

\subsection{Holomorphic Lie algebroids and associated smooth Lie algebroids}

Let now $X$ be a complex manifold, and $\belT_X$ its holomorphic tangent bundle, whose dual is $\Omega_X$, 
the bundle of holomorphic $1$-forms.

\begin{definition}
A \emph{holomorphic Lie algebroid} is a triple $(\bella,a,\{\cdot,\cdot\})$ where 
\begin{itemize}
\item $\bella$ is a coherent locally free $\corO_X$-module; 
\item $\{\cdot,\cdot\}$ is a $\C_X$-Lie algebra structure on $\bella$;
\item the anchor $a: \bella\rightarrow \belT_X$ is a morphism of $\corO_X$-modules which is also a morphism of 
sheaves of $\C_X$-Lie algebras, satisfying the usual Leibniz rule for Lie algebroids.
\end{itemize}
\end{definition}

For a holomorphic Lie algebroid $\bella$, we denote by $\Omega^k_\bella$ the sheaf $\bigwedge^k \bella^*$,
and call its sections holomorphic $k$-$\bella$-forms. The holomorphic differential 
$\dc_\bella: \Omega_\bella^k \rightarrow \Omega_\bella^{k+1}$ is defined by the formula \eqref{eqn:differential_la}
with the appropriate modifications.

If $(\bella,a,\{\cdot,\cdot\})$ is a holomorphic Lie algebroid, we can consider $\bella$ as a smooth bundle 
over $X$. Then the anchor $a$ induces a map of smooth bundles $a_\R: \bella \rightarrow T_X$, and similarly 
the bracket $\{\cdot,\cdot\}$ induces a bracket on the smooth sections of $\bella$, that we denote by 
$\{\cdot,\cdot\}_\R$. We call $\bella_\R = (\bella,a_\R,\{\cdot,\cdot\}_\R)$ the 
\emph{real Lie algebroid associated to $\bella$}. Furthermore, $\bella$ has a natural almost complex structure, 
i. e. an endomorphism $J_\bella$ whose square is $-1$ and that makes the diagram
$$\xymatrix{
\bella \ar[d]^{J_\bella} \ar[r]^{a_\R} &T_X \ar[d]^{J_X} \\
\bella \ar[r]^{J_\bella} & T_X
}$$ 
commutative, where $J_X$ is the almost complex structure of $X$. So we have the splitting of the tensor product 
$\bella\otimes \C = \bella^{1,0} \oplus \bella^{0,1}$ according to the $\pm \sqrt{1}$ eigenvalues of 
$J_\bella$. Moreover, $\bella^{1,0}$ and $\bella^{0,1}$ are invariant for the bracket 
$\{\cdot,\cdot\}_\C$ (defined extending $\{\cdot,\cdot\}_\R$ by $\C$-bilinearity), so they are
real sub-Lie algebroids of $\bella\otimes \C$, i. e. they form a matched pair.

It was noticed in \cite{xu} that we can associate another smooth Lie algebroid to a holomorphic Lie algebroid $\bella$, that we denote by $\bella_h$. As we will see, the latter encodes the information on the holomorphic cohomology of $\bella$. 
\begin{proposition}
If $\bella$ is a holomorphic Lie algebroid, then $(T^{0,1}_X,\bella^{1,0})$ is naturally a matched pair, and we will denote by $\bella_h$ the twilled sum $\bella^{1,0} \bowtie T^{0,1}_X$.
\end{proposition}
The representation of $T^{0,1}_X$ on $\bella^{1,0}$ is given by the holomorphic structure of $\bella$, while the representation of $\bella^{1,0}$ on $T^{0,1}_X$ is given by:
$$
\nabla_{u}(V) = \text{pr}^{0,1}([a(u),V])
$$
for $u\in \Gamma(\bella^{1,0})$ and $V\in \Gamma(T^{0,1}_X)$.

We will call $\bella_h = \bella^{1,0}\bowtie T^{0,1}_X$ the \emph{canonical complex Lie algebroid} associated to the 
holomorphic Lie algebroid $\bella$.

\subsection{Cohomology of holomorphic Lie algebroids} \label{sec:cohomology}

Let $\bella$ be a holomorphic Lie algebroid. Define the \emph{holomorphic Lie algebroid cohomology} of 
$\bella$ by
$$
H^p_{hol}(\bella) = \bH^p(\Omega_\bella^\bullet, \dc_\bella),
$$
where $\bH$ denotes the hypercohomology of a complex of sheaves.

Let $\bella_h = \bella^{1,0} \bowtie T^{0,1}_X$ be the canonical complex Lie algebroid associated to $\bella$. 
The associated double complex that computes its cohomology is
$$
\belA^{p,q}_{\bella_h} = \belA^p_{\bella^{1,0}} \otimes \belA^{q}_{T^{0,1}_X}\ ,
$$
and we denote by $A^{p,q}_{\bella_h}$ the vector spaces of global sections.

Because of the $\bar{\partial}$-Poincar\'e lemma, for any $p$ the complex 
$(K^{p,q}, \bar{\partial})_{q\geq 0} $ is exact in positive degree, and the kernel of the first map 
$K^{p,0} \rightarrow K^{p,1}$ is $\Omega_\bella^p$. So we have the following generalization of classical
theorems (cf. \cite{xu} and \cite{ugo_volodya}):

\begin{theorem}
Let $\bella$ be a holomorphic Lie algebroid. Then we have the following isomorphisms:
\begin{enumerate}
\item (holomorphic De Rham)   $$H^p_{hol}(\bella) \iso H^p(\bella_h,\C)\ ;$$
\item (Dolbeault)
$$
H^q(X,\Omega^p_\bella) \iso H^q(A_{\bella_h}^{p,\bullet}, \bar{\partial}_{\bella_h}).
$$
\end{enumerate}
\end{theorem}

Remark that the second of the previous isomorphisms can be readily generalized to the case of coefficients
in a holomorphic vector bundle: if $\belE$ is a holomorphic vector bundle, there is a natural way to define
the associated Dolbeault $\bella$-complex $\belA^{p,q}_{\bella_h}(\belE) = \belA^{p,q}_{\bella_h} \otimes \belE$
with an operator 
$$\bar{\partial}_\belE \: : \: \belA^{p,q}_{\bella_h}(\belE) \rightarrow \belA^{p,q+1}_{\bella_h}(\belE)$$ 
defined by the holomorphic structure of $\belE$. Then there are isomorphisms
\begin{equation} \label{eqn:gen_dolbeault}
H^q(X; \Omega_\bella^p \otimes \belE) \iso H^q(A_{\bella_h}^{p,\bullet}(\belE), \bar{\partial}_\belE).
\end{equation} 
As a notation, we will write $H^{p,q}(\bella_h; \belE)$ for the cohomology groups 
$H^q(A_{\bella_h}^{p,\bullet}(\belE), \bar{\partial}_\belE)$.

\spazio
Let $\gotU=\{U_\alpha\}$ be a sufficiently fine open covering of $X$, such that we have an isomorphism between sheaf and $\check{\text{C}}$ech cohomology over it.

Consider the double complex
$$
K^{p,q} = \check{C}^q (\gotU, \Omega_\bella^p) 
$$
with differentials $\dc_\bella,\check{\delta}$; its associated total complex $(T_\bella^\bullet, \delta)$ computes the hypercohomology of $\Omega_\bella^\bullet$, so it computes $H^k(\bella,\C)$.

The filtration by columns of the total complex
$$
F^r T_\bella ^k = \bigoplus_{p+q = k,\ p \geq r} K^{p,q}
$$
induces a filtration in the Lie algebroid cohomology:
$$
F^pH^k(\bella,\C) = \image(H^k(F^p T_\bella^\bullet) \rightarrow H^k(\bella,\C)).
$$

The associated spectral sequence has $E_1$ and $E_2$ terms given by
$$
E_1^{p,q} = H^{q}(X, \Omega_\bella^p) \qquad E_2^{p,q} = H^p(H^q(X,\Omega_\bella^\bullet), \dc_\bella).
$$
In the classical case, when $\bella = \belT_X$ and $X$ is a compact K\"ahler manifold, the Hodge decomposition implies that  $E^{p,q}_1 = E^{p,q}_\infty$ and that the differential $d_1$ is zero. Let us rewrite this fact as:
\begin{lemma}
Let $X$ be a compact K\"ahler manifold, and $\phi_{i_0,\ldots,i_q}$ a closed $\check{\text{C}}$ech $q$-cochain of $\Omega^p_X$.

Then the $\check{\text{C}}$ech $q$-cochain $\dc \phi_{i_0,\ldots,i_q}$ is $\check{\text{C}}$ech-exact, i.e. $\dc \phi = \check{\delta} \tau$ for some $\tau \in \check{C}^{q-1}(\gotU,\Omega_X^{p+1})$.
\end{lemma}

Now, for a general holomorphic Lie algebroid $\bella$, there is no analogue of Hodge decomposition, so we do not have degeneration of the spectral sequence at the first step. Anyway, we can use this lemma to find a mild degeneration of the spectral sequence: remark that on functions $\dc_\bella$ coincides with the composition of the exterior differential $\dc$ with the dual of the anchor, i.e. $\dc_\bella(f) = a\dual(\dc f)$ for any $f\in\corO_X$. Then, since the differential $d_1$ coincides with $\dc_\bella$, the previous lemma implies:
\begin{lemma}
If $X$ is a compact K\"ahler manifold, then,
restricted to $p=0$, the differential $d_1$ of the spectral sequence $d_1: E_1^{0,q} \rightarrow E_1^{1,q}$ is zero.
\end{lemma}

This can be rephrased as the degeneration at the first step of the spectral sequence associated to the 
truncated complex
$$
\tilde{K}^{p,q} = \left\{ \begin{array}{lcc}
K^{p,q} & & p=0,1 \\
0 & & p>1
\end{array} \right. \  .
$$

In particular, this gives the following corollary:
\begin{corollary} \label{cor:filtration_cohomology}
Over a compact K\"ahler manifold $X$ there are isomorphisms
$$
F^1H^k(\bella,\C) \iso H^k(F^1T_\bella^\bullet)\ .
$$
for any $k \geq 0$.
\end{corollary}

Since in the next sections we will be interested in $F^1H^2(\bella;\C)$, we need to represent its elements explicitly.
The elements of $H^2(F^1T_\bella^\bullet)$ are represented by closed elements of $F^1T_\bella^2=K^{1,1} \oplus K^{2,0}$, that is, pairs $(\phi_{\alpha\beta}, Q_\alpha)$ satisfying the equations
\begin{equation} \label{eqn:differential}
\begin{array}{c}
(\check{\delta} \phi)_{\alpha\beta\gamma} =0,\\
\dc_\bella \phi_{\alpha\beta} = (\check{\delta}Q)_{\alpha\beta},\\
\dc_\bella Q_\alpha = 0;
\end{array} \end{equation}
while coboundaries of $F^1T^2_\bella$ are of the form $((\check{\delta}\eta)_{\alpha\beta} , \dc_\bella \eta_\alpha)$ for $\eta_\alpha \in K^{1,0}$.

Hence we have a natural projection
$$
H^2(F^1T_\bella^\bullet) \rightarrow H^1(X, \Omega_\bella^1)\ , \qquad [(\phi,Q)] \rightarrow [\phi]\ ,
$$
well defined by the fact that if $(\phi,Q) = \delta \eta$ then $\phi = \check{\delta} \eta$.

\spazio
One has functoriality properties: any holomorphic Lie algebroid morphism $\Psi: \bella \rightarrow \bella'$ yields pull back morphisms $H^p(\bella',\C) \rightarrow H^p(\bella,\C)$ and $F^pH^k(\bella',\C)\rightarrow F^pH^k(\bella,\C)$.

In particular, we can apply this to the anchor map $a$ of a holomorphic Lie algebroid $\bella$: this yields morphisms $H^p_{DR}(X,\C)\rightarrow H^p(\bella,\C)$ and $F^pH^{k}(X,\C) \rightarrow F^pH^{k}(\bella,\C)$.

\subsection{Holomorphic $\bella$-connections}
Let now $\bella$ be a holomorphic Lie algebroid over a compact K\"ahler manifold $X$, and $\belE$ a holomorphic vector bundle on $X$. Similarly to the smooth case, we have:
\begin{definition}
A holomorphic $\bella$-connection on $\belE$ is a map of sheaves 
$\nabla: \belE\rightarrow \belE\otimes \Omega_\bella$ satisfying the Leibniz rule $\nabla(fe) = f\nabla e + e \otimes \dc_\bella f$ for any $f\in\corO_X$ and $e\in \belE$.
\end{definition}

The curvature $F_\nabla\in H^0(X,\cEnd (\belE) \otimes\Omega_\bella^2)$ of a holomorphic $\bella$-connection $\nabla$ is defined in the same way as in the smooth case. 

\spazio
Let $\bella$ be a holomorphic Lie algebroid and consider the associated Lie algebroid $\bella_h = \bella^{1,0}\bowtie T^{0,1}_X$. 
Let $\belE$ be a smooth vector bundle and $\nabla$ a $\bella_h$-connection on it. Since $\bella_h = \bella^{1,0} \oplus T^{0,1}_X$ as a vector bundle, $\nabla$ splits in two operators
$$
\nabla' : \Gamma(\belE) \rightarrow \Gamma(\belE) \otimes A^1_{\bella^{1,0}}\ , \qquad 
\nabla'': \Gamma(\belE) \rightarrow \Gamma(\belE) \otimes A^{0,1}_X\ ,
$$
satisfying the Leibniz rules 
$$
\nabla'(fs) = f \nabla'(s) + a_{\bella^{1,0}}(f) \otimes s\ , \qquad
\nabla''(fs) = f \nabla'' s + \bar{\partial}f \otimes s\ ,
$$
for $f\in C^\infty(X)$ and $s\in \Gamma(E)$.

The following lemma is straightforward:
\begin{lemma} \label{lem:connections}
Let $\bella$ be a holomorphic Lie algebroid over $X$, $E$ a smooth vector bundle on $X$ and $\nabla$ a 
smooth $\bella_h$ connection on $\belE$. Let $F$ be the curvature of $\nabla$, and 
$F = F^{2,0} + F^{1,1} + F^{0,2}$ according to the splitting of $A^2_{\bella_h}(\End \belE)$.

Then
\begin{enumerate}
\item $F^{0,2} = 0$ if and only if $\nabla''$ defines a holomorphic structure on $E$;
\item $F^{1,1} + F^{0,2} = 0$ if and only if $\nabla'$ is induced by a holomorphic $\bella$-connection,
where $E$ is equipped with the holomorphic structure defined by $\nabla''$;
\item $F = 0$ if and only if $\nabla$ is induced by a flat holomorphic $\bella$-connection.
\end{enumerate}
\end{lemma}

\spazio
We now study the problem of existence of holomorphic $\bella$-connections over a holomorphic vector bundle $\belE$. For $\bella=\belT_X$ the problem is well known, and we refer to \cite{atiyah} for further details: let $\belJ^1_X(\belE)$ be the bundle of holomorphic first order operators on $\belE$ with scalar symbol. It admits naturally a Lie algebroid structure, whose anchor is the symbol and whose bracket is the commutator of differential operators. It is usually called the \emph{Atiyah Lie algebroid} of $\belE$. There is a natural short exact sequence
\begin{equation} \label{atiyah_seq}
0 \rightarrow \cEnd(\belE) \rightarrow \belJ^1_X(\belE) \rightarrow \belT_X \rightarrow 0\ .
\end{equation}

A holomorphic ($\belT_X$-)connection on $\belE$ is equivalent to a splitting of this exact sequence, so there exists a holomorphic connection on $\belE$ if and only if the extension class 
$$
\mathfrak{a}(\belE)\in\Ext^1(\belT_X,\cEnd(\belE)) = H^1(X, \cEnd(\belE) \otimes \Omega_X )
$$ 
is zero. It is a theorem of Atiyah that the class $\mathfrak{a}(\belE)$ generates the characteristic ring of $\belE$, so one has:
\begin{theorem}
Let $\belE$ be a holomorphic vector bundle over a compact K\"ahler manifold $X$.

Then $\belR(\belE) = 0$ is a necessary condition for the existence of a holomorphic connection on $\belE$.
\end{theorem}
Let us recall how one can prove this: fix a Hermitian metric on $\belE$. Let $\nabla$ be the Hermitian connection on $\belE$; by definition it is the unique connection compatible with both the metric and the holomorphic structure of $\belE$. Its curvature $F$ is of type $(1,1)$. The cohomology class $[F] \in H^{1,1}(X, \cEnd(\belE))$ does not depend on the choice of the metric, and via the generalized Dolbeault isomorphism 
$$
H^{1,1}(X, \cEnd(\belE) ) \iso H^1(X, \cEnd(\belE) \otimes \Omega_X) \ ,
$$ 
$[F]$ corresponds to the class $\gota(\belE)$. The theorem then follows since, by definition, $\belR(\belE)$ is generated by $F$.

Let now $\bella$ be a holomorphic Lie algebroid, $\bella_h = \bella^{1,0} \bowtie T^{0,1}_X$, and $\belE$ a holomorphic vector bundle.

Let 
$$
\mathfrak{a}_\bella (\belE) = a^*(\mathfrak{a}(\belE)) \in \Ext^1(\bella, \cEnd(\belE)) \iso H^1(X, \cEnd(\belE)\otimes \Omega_\bella)
$$ 
be the pullback via the anchor $a$ of the Atiyah class of $\belE$. We have the corresponding diagram:
$$ \xymatrix{
0 \ar[r] &  \cEnd(\belE)  \ar[r] \ar[d] & \belJ^1_\bella(\belE) \ar[r]\ar[d] & \bella  \ar[d]\ar[r] & 0\\
0 \ar[r] & \cEnd(\belE) \ar[r] & \belJ^1_X(\belE) \ar[r] & \belT_X \ar[r] & 0 ,
}$$
where $\belJ^1_\bella(\belE)$ is the bundle of $1$-$\bella$-jets of $\belE$ that we are going to define soon in a broader context. For the moment, it can be considered just as the pull-back of $\belJ^1_X(\belE)$ via the anchor.

A $\bella$-connection on $\belE$ is equivalent to a splitting of the upper row, so it exists if and only if $\mathfrak{a}_\bella(\belE) = 0$.
We have the following:
\begin{proposition} \label{atiyah_per_lie_algebroid}
Let $\bella$ be a holomorphic Lie algebroid and $\belE$ a holomorphic vector bundle over a 
compact K\"ahler manifold $X$.

Then the class $\gota_\bella(\belE)$ generates the $\bella_h$-characteristic ring $\belR_{\bella_h}(\belE)$.
\end{proposition}
\begin{proof}
By the $\bella$-Dolbeault Theorem with coefficients, we have a commutative diagram
$$
\xymatrix{
H^1(X , \cEnd(\belE)\otimes \Omega_X ) \ar[r] \ar^{\iso}[d] & H^1(X , \cEnd(\belE)\otimes \Omega_\bella )\ar^{\iso}[d] \\
H^{1,1}(X,\cEnd(\belE)) \ar[r] & H^{1,1} (\bella_h , \cEnd(\belE))\ .
}
$$
Consider the images of $\gota(\belE),\gota_\bella(\belE)$ living in the lower row.

In general there are not inclusions of $H^{k,k}(\bella_h,\C)$ in $H^{2k}(\bella_h,\C)$, so Atiyah's argument does not apply straightforwardly. To obtain the asserion, we need, for each invariant polynomial $P$, the existence of a $\dc_{\bella_h}$-closed representative of the class $P(\gota_\bella(\belE)) \in H^{k,k}(\bella_h,\C)$. But $P(\gota_\bella(\belE))$ is the pullback of $P(\gota(\belE))$, and, since $X$ is a smooth projective variety, we can choose a $\dc$-closed representative of $P(\gota(\belE))$, whose pullback is a $\dc_\bella$-closed representative of $P(\gota_\bella(\belE))$.
\end{proof}

\spazio
It is well known that if a coherent $\corO_X$-module $\belE$ admits a $T_X$-connection then it is locally free. This is no more true in the Lie algebroid case: if $\bella$ is a holomorphic Lie algebroid and $\belG$ the associated holomorphic foliation, $\belE$ a coherent $\corO_X$-module and $\nabla$ a holomorphic $\bella$-connection on $\belE$, then we can only say that $\belE_{|G}$ is locally free for any leaf $G$ of $\belG$ (see \cite{fernandes}). Now we want to study which of the previous results generalize to the case when $\belE$ is not locally free.

Let $\bella$ be a holomorphic Lie algebroid. As in \cite{calaque},
we can introduce the sheaf of first $\bella$-jet bundle as follows:
set $\belJ^1_\bella = \Omega_\bella \oplus \corO_X$ with the product 
$(\alpha,f)(\beta,g) = (f\beta + g\alpha , fg)$. Then $\Omega_\bella$ is a sheaf of ideals of $\belJ^1_\bella$, and we have the exact sequence
$$
0 \rightarrow \Omega_\bella \rightarrow \belJ^1_\bella \rightarrow \corO_X \rightarrow 0 \ .
$$

For $\belE$ a coherent $\corO_X$-module, over the sheaf $\belJ^1_\bella(\belE) = \Omega_\bella \otimes \belE \oplus \belE$ we can naturally define a left and right $\belJ^1_\bella$-module structure. These induce a left and a right $\corO_X$-module structures on $\belE$: the right $\corO_X$-module structure on $\belJ^1_\bella(\belE)$ coincide with the $\corO_X$-module structure induced by the direct sum, but the left $\corO_X$-module structure will in general be different. We have the exact sequence
$$
0\rightarrow \Omega_\bella \otimes \belE \rightarrow \belJ^1_\bella(\belE) \rightarrow \belE \rightarrow 0
$$
of both left and right $\corO_X$ and $\belJ^1_\bella$-modules. Remark that when $\bella = \belT_X$ and $\belE$ is locally free, this is equivalent to the sequence \eqref{atiyah_seq}. We define the $\bella$-Atiyah class of the coherent $\corO_X$-module $\belE$ as the class of this extension as left $\corO_X$-modules in $\Ext_{\corO_X}^1(\belE,\belE\otimes \Omega_\bella)$. It is zero if and only if there exists a holomorphic $\bella$-connection on $\belE$. We have:
\begin{proposition}
Let $X$ be a compact K\"ahler manifold, and $\bella$ a holomorphic Lie algebroid over it.

Then the $\bella$-Atiyah class of any coherent sheaf $\belE$ generates $\belR_\bella(\belE)$.
\end{proposition}
\begin{proof}
One can check that when $\bella = \belT_X$ this definition coincides with the usual definition of Atiyah class of a coherent sheaf (see for example \cite{huybrechts_lehn}, chapter 10); in particular, the proposition is well known in the case $\bella = \belT_X$. Then one can conclude following the same arguments as in Proposition 
\ref{atiyah_per_lie_algebroid}.
\end{proof}

We now want to prove the following: 
\begin{proposition} \label{cor:vanishing_characteristic_ring}
Let $\bella$ be a holomorphic Lie algebroid over a compact K\"ahler manifold $X$, and $\belE$ a torsion free $\corO_X$-module.

If $\nabla$ is a holomorphic $\bella$-connection on $\belE$, then $[\trace (F_\nabla)] = 0$.
\end{proposition}

If $\belE$ is locally free, this is straightforward, since the class of the trace of the curvature of $\nabla$ is the first $\bella$-Chern class of $\belE$. With the following, we prove that this is true also if $\belE$ is torsion free:

\begin{lemma}
Let $\belE$ be a torsion free $\corO_X$-module and $\nabla$ a holomorphic $\bella$-connection on it.

Then $[\trace (F_\nabla)] = c_{1,\bella} (\belE)$.
\end{lemma}
\begin{proof}
Let $U\subseteq X$ be the open subset where $\belE$ is locally free. Then since $\belE$ is torsion free, its complement has codimension at least $2$. Consider $\det(\belE)$, the determinant line bundle of $\belE$: on the open $U$ where it is locally free we have
$$
\det(\belE)_{|U} \iso \bigwedge^r \belE_{|U}\ ,
$$
where $r$ is the rank of $\belE$. Since $c_{1,\bella_h}(\belE)$ is the pull back of $c_1(\belE)$ and 
$c_1(\belE) = c_1(\det(\belE))$, we have $c_{1,\bella_h}(\belE) = c_{1,\bella_h}(\det(\belE))$.

Over $U$, $\nabla$ induces a holomorphic $\bella$-connection $\tilde{\nabla}$ on $\det{\belE}$, defined as the $r$th exterior power of $\nabla$. Now, since the complement of $U$ has codimension at least $2$, we can extend $\tilde{\nabla}$ to the whole $X$ by Hartogs lemma: for any $s\in \det(\belE)_{|U}$, $\sigma\in (\det(\belE)_{|U})^*$ and $u\in \bella$, the function
$$
f_{s,\sigma,u} = \langle \sigma, \tilde{\nabla}_u (s) \rangle
$$
is holomorphic on $U$, so it extends uniquely to $X$. So we can define $\tilde{\nabla}_u(s) = f_{s,u} \in \det(\belE)$, since $\sigma \mapsto f_{s,\sigma,u}$ is a $\corO_X$-linear map.

Let $\tilde{F}$ be the curvature of $\tilde{\nabla}$. We claim that $\tilde{F} = \trace(F_\nabla)$:
this is clearly true over $U$, so by the previous extension argument this holds on the whole $X$. 

So we finally have
$$
c_{1,\bella_h}(\belE) = c_{1,\bella_h}(\det(\belE)) = [\tilde{F}] = [\trace (F_\nabla)]\ ,
$$
as desired.
\end{proof}

\subsection{Example: logarithmic connections}

Let $X$ be a smooth projective variety and $D$ an effective normal crossing divisor on it. Consider the sheaf $\Omega_X(\log D)$ of meromorphic $1$-forms with logarithmic poles along $D$. This is the locally free $\corO_X$-module locally generated  by $\frac{\dc x_{1}}{x_{1}}, \ldots, \frac{\dc x_{t}}{x_{t}},  \dc x_{t+1}, \ldots , \dc x_{n}$, where $x_1 ,\ldots ,x_n$ are local coordinates on $X$ such that $D$ has equation $x_{1} \cdots x_{t} =0$ in this coordinates. Remark that we have a natural inclusion $\Omega_X \hookrightarrow \Omega_X(\log D)$, and the quotient is isomorphic to the structure sheaf of $\tilde{D}$, the normalization of $D$. The exterior differential extends naturally to $\dc : \Omega^p_X(\log D) \rightarrow \Omega^{p+1}_X(\log D)$, where $\Omega_X^p(\log D) = \bigwedge^p \Omega_X(\log D)$.

Consider $\belT_X(\log D)$, the dual of $\Omega_X(\log D)$. Dualizing the sequence
$$
0 \rightarrow \Omega_X \rightarrow \Omega_X(\log D) \rightarrow \corO_{\tilde{D}} \rightarrow 0
$$
we obtain an inclusion $\belT_X(\log D) \hookrightarrow \belT_X$. This subsheaf is locally free and closed under the commutator of vector fields, so it inherits a holomorphic Lie algebroid structure. 

The holomorphic Lie algebroid cohomology $H_{hol}^p(\belT_X(\log D))$ is equal, just by definition, to the hypercohomology of the logarithmic deRham complex $\bH^p(X,\Omega^\bullet_X(\log D))$, that is well known to be isomorphic to the deRham cohomology $H^p_{DR}(U,\C)$ of the open $U=X\setminus D$.

\spazio
Starting from important works of Deligne, holomorphic integrable connections with logarithmic poles along a divisor $D$ have been extensively studied, see for example \cite{deligne}, \cite{esnault}. From our point of view, an integrable connection with logarithmic poles along $D$ on a coherent sheaf $\belE$ is just a flat $\belT_X(\log D)$-connection on $\belE$.

In the Appendix B of \cite{esnault}, it is shown how one can compute the Chern classes of a holomorphic vector bundle in terms of a connection with logarithmic poles along $D$. In particular, it is shown:
\begin{theorem} \label{thm:chern_logarithmic}
Let $X$ be a smooth projective variety and $D$ an effective normal crossing divisor on $X$, $\belE$ a coherent $\corO_X$-module and $\nabla$ a $T_X(\log D)$-connection on $\belE$.
Let $D = \sum a_i D_i$ with $D_i$ irreducible, and $[D_i]$ the class of $D_i$ in $H^2(X,\C)$.

Then the $p$-th Chern class of $\belE$ is a $\C$-linear combination of $[D_{i_1}]^{k_1} \cdots [D_{i_t}]^{k_t}$ with $\sum k_\alpha = p$.
\end{theorem}

Now we can easily see that this theorem implies Theorem \ref{cor:vanishing_characteristic_ring} when $\bella = \belT_X(\log D)$: since 
$$
H^p((\belT_X(\log D))_h , \C) \iso H^p_{hol}(\belT_X(\log D)) \iso H^p_{DR}(U,\C)\ ,
$$
the pullbacks of $[D_i]$ to the Lie algebroid cohomology vanish. So if $\belE$ admits a holomorphic connection with logarithmic poles along $D$, its $\belT_X(\log D)$-Chern classes vanish.

\section{Sheaves of filtered algebras}

\subsection{Lie algebroids associated to an almost polynomial filtered algebra}

In this section we want to classify the sheaves of rings $\Lambda$ satisfying the axioms of Simpson's paper \cite{simpson_representation_1}.

Let $X$ be a smooth algebraic variety over $\C$ or a complex manifold. By a \emph{sheaf of filtered algebras on $X$} we shall mean what Simpson calls "almost polynomial sheaf of rings of differential operators", that is a sheaf of rings $\Lambda$ over $X$ with a filtration of subsheaves of abelian subgroups $\Lambda_{(i)} \subseteq \Lambda_{(i+1)} \subseteq \ldots \subseteq \Lambda$ for $i\in \Z_{\geq0}$,  satisfying the following axioms:
\begin{enumerate}
\item $\C_X$, the constant sheaf over $X$ is in the center of $\Lambda$;
\item $\Lambda_0 =\corO_X$, $\Lambda_{(i)}\Lambda_{(j)} \subseteq \Lambda_{(i+j)}$ and $\Lambda = \bigcup \Lambda_{(i)}$; this implies that $\corO_X$ is a sheaf of subrings of $\Lambda$, and each $\Lambda_{(i)}$ carries an $\corO_X$-bimodule structure;
\item the left and right $\corO_X$-module structures on $\Gr_i\Lambda  = \Lambda_{(i)} / \Lambda_{(i-1)}$ coincide;
\item the graded $\corO_X$-modules $\Gr_i \Lambda$ are coherent;
\item the graded algebra $\Gr_*\Lambda$ is isomorphic to the symmetric algebra over the first graded piece $\Gr_1\Lambda$.
\end{enumerate}

By a \emph{splitting} of $\Lambda$ we mean a left $\corO_X$-module morphism $\zeta: \Gr_1\Lambda \rightarrow \Lambda_{(1)}$ that splits the exact sequence
\begin{equation} \label{eqn:ses}
0\rightarrow \corO_X \rightarrow \Lambda_{(1)} \rightarrow \Gr_1\Lambda \rightarrow 0.
\end{equation}

\spazio
Since the graded object of $\Lambda$ is commutative, the top degree part of the commutator of two elements vanishes. So $[\Lambda_{(i)} , \Lambda_{(j)}] \subseteq \Lambda_{(i+j-1)}$, and this allows us to define the following holomorphic Lie algebroid structures on $\Lambda_{(1)}$ and $\Gr_1\Lambda$:
\begin{itemize}
\item for each $x\in \Lambda_{(1)}$, the map $a_\Lambda(x):f \rightarrow [x,f] = xf-fx$ is a derivation of $\corO_X$, yielding a map $\Lambda_{(1)} \rightarrow \belT_X$; 
\item since $a_\Lambda(x+g) = a_\Lambda(x)$ for any $g\in\corO_X$, $a_\Lambda$ factors through the quotient $\Lambda_{(1)} \rightarrow \Gr_1 \Lambda$, and we have another anchor  $a_\text{G}: \Gr_1\Lambda \rightarrow \belT_X$;
\item $\Lambda_{(1)}$, is closed under the commutator $[\cdot,\cdot]$, and together with the anchor $a_\Lambda$ it defines a Lie algebroid structure on $\Lambda_{(1)}$;
\item for each $u,v\in\Gr_1\Lambda$, define 
$$
[u,v]_\text{G} = [x,y] \: \text{mod} \: \corO_X
$$ 
where $x,y \in\Lambda_{(1)}$ are representatives of $u$ and $v$ respectively; one can check that this definition does not depend on the choice of $x$ and $y$ in their classes, and that the thus defined bracket satisfies the Jacobi identity and the Leibniz rule w. r. t. the anchor $a_\text{G}$.
\end{itemize}

We call $(\Gr_1\Lambda,a_\text{G},[\cdot,\cdot]_\text{G})$ the Lie algebroid associated to $\Lambda$.

It follows from the definition that the projection $\Lambda_{(1)}\rightarrow \Gr_1\Lambda$ is a Lie algebroid map, so we can look at \eqref{eqn:ses} as an exact sequence of Lie algebroids, where $\corO_X$ is given the trivial Lie algebroid structure; hence $\Lambda_{(1)}$ is an abelian Lie algebroid extension of $\Gr_1\Lambda$ by $\corO_X$.

Similarly to Sridharan's paper \cite{sridharan}, in order to classify the $\Lambda$'s we proceed as follows: we first classify abelian Lie algebroid extensions of a holomorphic Lie algebroid $\bella$ by $\corO_X$, and then see that such extensions are in a one to one correspondence with the isomorphism classes of pairs $(\Lambda,\Xi)$, where $\Lambda$ is a sheaf of filtered algebras on $X$ and $\Xi: \Gr_\bullet \Lambda \rightarrow \Sym^\bullet_{\corO_X} \bella$ is an isomorphism of sheaves of graded algebras.

\subsection{Lie algebroid extensions}
Let $X$ be a complex manifold, and
\begin{equation} \label{eqn:la_ext}
0\rightarrow \corO_X \rightarrow \bella' \rightarrow \bella \rightarrow 0
\end{equation}
an abelian extension of holomorphic Lie algebroids over it.

Assume that there exists a global splitting $\zeta : \bella \rightarrow \bella'$ of the sequence considered as a sequence of left $\corO_X$-modules. This gives an isomorphism of left $\corO_X$-modules $\hat{\zeta}: \bella' \rightarrow \corO_X \oplus \bella$, so we can write
$$
[f+u,g+v]_{\bella'} = [u,v]_{\bella} + Q(u,v) + a_{\bella}(u)(g) - a_{\bella}(v)(f)
$$
with $Q(u,v) \in \corO_X$. $Q$ is antisymmetric and $\corO_X$-bilinear, so it is a holomorphic $2$-$\bella$-form; moreover, one can check that the Jacobi identity for $[\cdot,\cdot]_{\bella'}$ is satisfied if and only if $\dc_{\bella} Q = 0$.

It is then easy to see that by changing the splitting, we change $Q$ by an exact holomorphic $2$-$\bella$-form. More precisely, we have the following:
\begin{lemma} \label{lem:change_split}
Let $\zeta_1,\zeta_2$ be two global left $\corO_X$-module splittings of 
$$
0 \rightarrow \corO_X \rightarrow \bella' \rightarrow \bella \rightarrow 0\ ,
$$
and let $\psi = \zeta_2-\zeta_1 \in \Hom(\bella, \corO_X) = H^0(X, \Omega_\bella)$. 
Let $Q_1,Q_2$ be the closed holomorphic $2$-$\bella$-forms associated to $\zeta_1$, $\zeta_2$ respectively.

Then $Q_2-Q_1 = \dc_{\bella} \psi$.
\end{lemma}

Because of this, remembering the notation of Section \ref{sec:cohomology}, since
$$
\frac{H^0(X,\Omega^2_\bella)_{\text{closed}}}{\dc_\bella (H^0(X, \Omega_\bella))} = E^{2,0}_2 \iso F^2H^2(\bella,\C)\ ,
$$
we have
\begin{corollary} \label{cor:global_split}
Let $X$ be a complex manifold and $\bella$ a holomorphic Lie algebroid over it. Then the isomorphism classes of holomorphic Lie algebroid extensions
$$
0\rightarrow \corO_X \rightarrow \bella' \rightarrow \bella \rightarrow 0
$$
which are split as sequences of left $\corO_X$-modules are in a one to one correspondence with the elements of $F^2H^2(\bella,\C)$.
\end{corollary}

\spazio
Now we examine what happens if the extension \eqref{eqn:la_ext} does not split as a sequence of $\corO_X$-modules: let $\Phi \in \Ext^1(\bella, \corO_X) \iso H^1(X,\Omega_\bella)$ be the associated cohomology class. For a sufficiently nice open covering $\gotU= \{U_\alpha\}$ of $X$ we can represent $\Phi$ by a closed $1$-$\check{\text{C}}$ech-cocycle $\phi_{\alpha\beta}$, and choose local splittings 
$$
\zeta_\alpha : \bella_{|U_\alpha} \rightarrow \bella'_{|U_\alpha}
$$
satisfying $\zeta_\beta - \zeta_\alpha = \phi_{\alpha\beta}$. Then we can do the previous construction on each $U_\alpha$ and obtain a closed holomorphic $2$-$\bella$-form $Q_\alpha \in H^0(U_\alpha, \Omega^2_\bella)$. Because of Lemma \ref{lem:change_split}, these satisfy $Q_\beta - Q_\alpha = \dc_\bella \phi_{\alpha\beta}$ on the double overlaps $U_{\alpha\beta}$. So the pair $(Q_\alpha,\phi_{\alpha\beta}) $ is a closed element of $F^1T_\bella^2$.

Now let $\phi'_{\alpha\beta}$ be another representative of $\Phi$; then $\phi'_{\alpha\beta} - \phi_{\alpha\beta} = (\check{\delta}\eta)_{\alpha\beta}$ for some $\eta \in \check{C}^1(\gotU, \Omega_\bella^1)$. Let $\zeta'_\alpha$ be local splittings over $U_\alpha$ satisfying $\zeta'_\beta - \zeta'_\alpha = \phi'_{\alpha\beta} = \phi_{\alpha\beta} + (\check{\delta}\eta)_{\alpha\beta}$, and $Q'_\alpha \in H^0(U_\alpha, \Omega^2_\bella)$ the closed $2$-$\bella$-forms associated to the splittings $\zeta'_\alpha$. Then $\check{\delta}(Q'-Q)_{\alpha\beta} = \dc_\bella (\check{\delta}\eta)_{\alpha\beta}$, which means, since $\check{\delta}$ and $\dc_\bella$ commute, that the local $2$-$\bella$-forms $Q'_\alpha- Q_\alpha - \dc_\bella \eta_\alpha$ glue to a global $2$-$\bella$-form $G$. So
$$
(Q'_\alpha, \phi'_{\alpha\beta}) - (Q_\alpha, \phi_{\alpha\beta}) = \delta (\eta_\alpha + G_{|U_{\alpha}}),
$$
Hence the cohomology class of $(Q_\alpha,\phi_{\alpha\beta})$ in $H^2(F^1T^\bullet)$ is independent of the choices we made. 

Now, if $X$ is a compact K\"ahler manifold, by Corollary \ref{cor:filtration_cohomology} we can identify $H^2(F^1T^\bullet)$ with $F^1H^2(\bella;\C)$, and we obtain:
\begin{theorem}
Let $\bella$ be a holomorphic Lie algebroid over a compact K\"ahler manifold $X$.

Then the isomorphism classes of abelian extensions of $\bella$ by $\corO_X$ are in a one to one correspondence with the elements of the cohomology group $F^1H^2(\bella,\C)$.
\end{theorem}

In this interpretation, the map $F^1H^2(\bella,\C) \rightarrow H^1(X,\Omega_\bella)$ associates to an extension of Lie algebroids its class as an extension of $\corO_X$-modules, while Corollary \ref{cor:global_split} describes the fiber of this map over $0$.

\subsection{From extensions to algebras}

In this section, we will modify Sridharan's construction of twisted enveloping algebras to the case of Lie algebroids, and see that the datum of an abelian extension of a holomorphic Lie algebroid $\bella$ by $\corO_X$ is equivalent to a pair $(\Lambda,\Xi)$ with $\Lambda$ a sheaf of filtered algebras over $X$ and $\Xi:\Gr \Lambda \rightarrow \Sym^\bullet \bella$ an isomorphism of sheaves of graded algebras.

Let 
$$
0 \rightarrow \corO_X \rightarrow \bella' \rightarrow \bella \rightarrow 0
$$
be an abelian extension of holomorphic Lie algebroids and $\sigma = (Q_\alpha, \phi_{\alpha\beta})$ a representative of the class $\Sigma \in F^1H^2(\bella,\C)$ associated to this extension.

For each $\alpha$ define the following sheaf of algebras over $U_\alpha$:
$$
\belU_\sigma(\bella)_\alpha = T_{\C_{U_\alpha}}^\bullet(\corO_{U_\alpha} \oplus \bella_{|U_\alpha})/I_{Q_\alpha}
$$
where $T_{\C_{U_\alpha}}^\bullet$ denotes the full $\C_{U_\alpha}$-tensor algebra, and $I_{Q_\alpha}$ is the ideal sheaf generated by the elements of the form 
$$
f \otimes (g + u) - fg + fu \qquad \text{for}\ \  f,g\in \corO_{U_\alpha},\ \ u \in \bella_{|U_\alpha} 
$$ $$
(f+u)\otimes (g+v) - (g+v)\otimes (f+u) - [f+u,g+v]' - Q_\alpha(u,v)\id \quad \text{for}\ \  u,v \in \bella_{|U_\alpha}\ ,
$$ 
where $[\cdot,\cdot]'$ denotes the bracket on $\corO_X \oplus \bella$ corresponding to the trivial extension, that is
$$
[f+u,g+v] = a(u)(g) - a(v)(f) + [u,v]\quad \text{for} \ f,g\in \corO_X,\ u,v\in\bella\ .
$$

Now, on double overlaps $U_{\alpha\beta}$ define maps
$$
T_{\C_{U_{\alpha\beta}}}^\bullet(\corO_X \oplus \bella_{|U_{\alpha\beta}}) \rightarrow T_{\C_{U_{\alpha\beta}}}^\bullet(\corO_X \oplus \bella_{|U_{\alpha\beta}})
$$ 
by 
$$
f + u \rightarrow f + \phi_{\alpha\beta}(u) + u \qquad  f\in \corO_{U_{\alpha\beta}},\  u\in\bella_{|U_{\alpha\beta}}\ .
$$
One can check that this map descends to an isomorphism of sheaves of algebras
$$
g_{\alpha\beta}:(\belU_\sigma(\bella)_\alpha)_{|U_{\alpha\beta}}\rightarrow (\belU_\sigma(\bella)_\beta)_{|U_{\alpha\beta}} .
$$
Since $\phi_{\alpha\beta}$ is $\check{\delta}$-closed, we have $g_{\alpha\beta}g_{\beta\gamma}g_{\gamma\alpha} = \id$ on the triple intersections, so we can glue the local sheaves $\belU_\sigma(\bella)_\alpha$ via the isomorphisms $g_{\alpha\beta}$, and obtain a sheaf of algebras $\belU_\sigma(\bella)$ on $X$.

Since the tensor algebra is naturally graded, $\belU_\sigma(\bella)$ inherits a filtration. Moreover, since the ideals $I_{Q_\alpha}$ are generated by elements of the form "commutator" + "lower degree terms", we have that $\Gr \belU_\sigma(\bella) \iso \Sym^\bullet \bella$. Using the injections $\bella_{|U_\alpha} \rightarrow T_{\corO_{U_\alpha}}^\bullet \bella_{|U_\alpha}$, one can construct one such isomorphism explicitly, that we shall denote $\Xi_\sigma$.

Finally, if $\sigma'$ is another representative of $\Sigma$ and $\sigma'-\sigma = \delta(\eta)$, it is possible to construct from $\eta$ an isomorphism of  $\belU_{\sigma'}(\bella) \rightarrow \belU_\sigma(\bella)$ commuting with the isomorphisms $\Xi_{\sigma'}$ and $\Xi_\sigma$.

Summing up, we have
\begin{theorem}
Let $X$ be a compact K\"ahler manifold and $\bella$ a holomorphic Lie algebroid on it.

Then there is a 1-to-1 correspondence between
\begin{itemize}
\item abelian extensions of $\bella$ by $\corO_X$;
\item elements of the vector space $F^1H^2(\bella,\C)$;
\item isomorphism classes of pairs $(\Lambda,\Xi)$, where $\Lambda$ is a sheaf of filtered algebras on $X$ whose associated Lie algebroid is $\bella$, and $\Xi: \Gr \Lambda \rightarrow \Sym^\bullet \bella$ an isomorphism of sheaves of graded algebras.
\end{itemize}
\end{theorem}

\spazio
In particular, if $\bella$ is a holomorphic Lie algebroid and $\Sigma \in F^1H^2(\bella,\C)$, we denote the associated sheaf of filtered algebras by $\Lambda_{\bella,\Sigma}$.

This construction is functorial: if $\bella, \bella'$ are two holomorphic Lie algebroids and $\Psi:\bella\rightarrow \bella'$ is a Lie algebroid morphism, we have an induced pull back morphism between the cohomologies $\Psi^*: H^p(\bella',\C)\rightarrow H^p(\bella,\C)$ preserving the filtration.

Then it is easy to show the following:
\begin{lemma} \label{lem:functoriality}
Let $\Psi:\bella\rightarrow \bella'$ be a morphism of holomorphic Lie algebroids and $\Sigma \in F^1H^2(\bella',\C)$. Then $\Psi$ extends to a morphism of sheaves of filtered algebras
$$
\Psi: \Lambda_{\bella, \Psi^*\Sigma} \longrightarrow \Lambda_{\bella',\Sigma}.
$$
\end{lemma}

\subsection{Examples: algebras associated to the canonical Lie algebroid}

Let $\bella = \belT_X$ be the holomorphic tangent bundle with the canonical Lie algebroid structure. 
Remark that in this case we have Hodge decomposition, so $F^1H^2(\belT_X,\C) = F^1H^2_{DR}(X,\C) = H^{2,0}(X) \oplus H^{1,1}(X)$. 

If $\Sigma=0$, then $\Lambda_{\belT_X,0}$ is the sheaf of algebras of holomorphic differential operators $\corD_X$.

\spazio
For $\Sigma = [(0,Q)]$ we can describe $\Lambda_{\belT_X,[(0,Q)]}$ in terms of local coordinates as follows. Let $x^1,\ldots, x^n$ be local holomorphic coordinates of $X$ and $\partial_{x^i}$ the corresponding frame of $\belT_X$. Let $Q_{ij}$ be such that $Q= \sum_{i,j}Q_{ij} \dc x^i \wedge \dc x^j$. 
Then the commutator of elements in $\Lambda_{\belT_X,[(0,Q)]}$ is determined by:
$$\begin{array}{l}
[x^i , x^j] = 0\ , \\
\left[ x^i, \partial_{x^j} \right] = \delta^i_j\  , \\
\left[\partial_{x_i} , \partial_{x_j} \right] = Q_{ij}\ .
\end{array}$$
This is the operator algebra corresponding to a magnetic monopole of charge $Q$.

\spazio
Another case that has an explicit description is when $[\phi] \in H^{1,1}(X) \cap H^2(X,\Z)$. Let $L$ be a holomorphic line bundle on $X$ given by transition functions $g_{\alpha\beta}:U_{\alpha\beta} \rightarrow \C^*$; this defines a class in $H^{1,1}(X) \cap H^2(X,\Z)$ represented, through Dolbeault isomorphism, by a cocycle $\phi_{\alpha\beta} = g_{\alpha\beta}^{-1} \dc g_{\alpha\beta}\ \in\  H^1(X,\Omega_X)$. For $Q=0$, the class $[(\phi,0)]$ defines the Atiyah Lie algebroid of $L$:
$$
0 \rightarrow \corO_X \rightarrow \belJ^{1}_X(L) \rightarrow \belT_X \rightarrow 0\ ,
$$
so we have $\Lambda_{\belT_X,[(\phi,0)]}\iso \corD (L)$, the algebra of differential operators on the Line bundle $L$.
One can then think of $\Lambda_{\belT_X,[(\phi_{\alpha\beta},Q)]}$ as the operator algebra of a twisted magnetic monopole with charge $Q$ and twisting line bundle $L$.

\subsection{Other examples}

Let now $\bella = \belK$ be a holomorphic bundle of Lie algebras. Then $\belK(U)$ is actually an $\corO_X(U)$-Lie algebra for each open $U\subseteq X$, and $H^p(H^0(U, \Omega^\bullet_\belK),\dc_\belK)$ coincides with the Chevalley-Eilenberg cohomology $H_{CE}^p(\belK(U),\corO_X(U))$.

For $\Sigma = 0$, the associated sheaf of filtered algebras $\Lambda_{\belK,0}$ is the sheaf of universal enveloping algebras of $\belK$: indeed, we have that $\Lambda_{\belK,0}(U)$ is the universal enveloping algebra of the $\corO_X(U)$-Lie algebra $\belK(U)$ for each open $U\subseteq X$. 

More generally, if $\Sigma = [(Q_\alpha,\phi_{\alpha\beta})]$, then $ \Lambda_{\belK,\Sigma} (U_\alpha)$ is the Sridharan's universal enveloping algebra of the $\corO_X(U_\alpha)$-Lie algebra $\belK(U_\alpha)$ associated to the class $Q_\alpha \in H^0(U_\alpha, \Omega_\belK^2) = H_{CE}^2(\belK(U_\alpha),\corO_X(U_\alpha))$. For any open $V\subseteq U_\alpha$, we obtain $\Lambda_{\belK,\Sigma}(V)$ in the same way using the restrictions $Q_{\alpha|V}$, so we have an explicit description of the sheaves $(\Lambda_{\belK,\Sigma})_{|U_\alpha}$. One can then check that $\phi_{\alpha\beta}$ give rise to isomorphisms on the overlaps.

\spazio
For any regular integrable holomorphic foliation $\corF \subseteq \belT_X$, we have a holomorphic Lie algebroid structure induced by the canonical one on $\belT_X$.

If $\Sigma = 0$, then $\Lambda_{\corF,0}$ is isomorphic to the algebra $\corD_\corF$ of differential operators along the foliation. For more general $\Sigma$ we obtain operator algebras of (twisted) monopoles propagating along the foliation.


\section{$\Lambda$-modules}

\subsection{Moduli spaces}

Up to now, we have worked mainly in a complex analytic setting, while from now on we will stay in the algebraic category. By the GAGA principle, we can identify the algebraic objects with the associated holomorphic ones, and we will mix the terminology, so that a holomorphic Lie algebroid over a smooth projective variety will be the same thing as an algebraic one. Actually, all the previous results can be formulated and proved in an algebraic setting, since the only transcendental result that we used is the Hodge decomposition, which is also true for smooth complex projective varieties.

\spazio

Let $\Lambda$ be a sheaf of filtered algebras over a complex smooth projective variety $X$, and $\belE$ a coherent sheaf on $X$.

\begin{definition}
A $\Lambda$-module structure on $\belE$ is an $\corO_X$-morphism $\mu: \Lambda \otimes \belE \rightarrow \belE$ satisfying the usual module axioms and such that the $\corO_X$-module structure on $\belE$ induced by $\corO_X \rightarrow \Lambda$ coincides with the original one. 
\end{definition}

We say (cf. \cite{simpson_representation_1}) that a $\Lambda$-module $(\belE,\mu)$ is \emph{(semi)stable} if $\belE$ is torsion free and for any subsheaf $\belF\subseteq \belE$ invariant under $\mu$ (i.e. such that $\mu(\Lambda\otimes \belF)\subseteq \belF$), one has $p(\belF) < p(\belE)$ (resp. $p(\belF)\leq p(\belE)$), where $p(\belE)$ is the reduced Hilbert polynomial of $\belE$, defined as the ratio $P(\belE)/rnk(\belE)$, where $P(\belE)(n) = \chi(\belE\otimes \corO_X(n))$ is the Hilbert polynomial of $\belE$, and $rnk(\belE)$ its rank.

The main result of \cite{simpson_representation_1} is the following:
\begin{theorem}    \label{thm:simpson_main}
Let $X$ be a smooth projective variety, $\Lambda$ a sheaf of filtered algebras on $X$ and $P$ a numerical polynomial.

Then there exists a quasi projective scheme $\moduli_\Lambda(P)$ that is a coarse moduli space for semistable $\Lambda$-modules with Hilbert polynomial $P$.
\end{theorem}

\spazio
Let $\bella$ be a holomorphic Lie algebroid, $\Sigma\in F^2H^2(\bella,\C)$ and $\Lambda = \Lambda_{\bella,\Sigma}$. 
So there exist left $\corO_X$-module splittings of the sequence
$$
0\rightarrow \corO_X \rightarrow \Lambda_{(1)} \rightarrow \bella \rightarrow 0 \ .
$$
Choose a splitting $\zeta$, which provides a representative $Q\in H^0(X,\Omega_\bella^2)_{\text{closed}}$ of $\Sigma$.

To a $\Lambda$-module structure $\mu$ on $\belE$, we canonically associate a sheaf map $\nabla: \belE \rightarrow \belE\otimes \bella\dual$, defined by
$$
\langle\nabla(e),v\rangle = \mu(\zeta(v)\otimes e)
$$ 
for all $v\in \bella$. This map satisfies the $\dc_\bella$-Leibniz rule $\nabla(fe) = f\nabla(e) + e \otimes \dc_\bella(f)$ for any $f\in\corO_X$ and $e\in \belE$, so it is a holomorphic $\bella$-connection on $\belE$.

Vice versa, if we have a holomorphic $\bella$-connection $\nabla$ on $\belE$, we can define a morphism $\mu_1: \Lambda_{(1)} \otimes \belE \rightarrow \belE$ by 
$$
\mu_1((f + \zeta(u))\otimes e) = fe + \langle\nabla(e),v\rangle.
$$ 
This morphism can be extended to a $\Lambda$-module structure if and only if 
$$
\mu_1(a \otimes \mu_1(b\otimes e)) - \mu_1(b \otimes \mu_1(a\otimes e)) = \mu_1([a,b] \otimes e)
$$ 
for any $a,b \in \Lambda_{(1)}$, $e\in\belE$. This condition is satisfied if and only if
$$\begin{array}{l}
\langle\nabla(e), [u,v]\rangle + Q(u,v)e - \langle\nabla(\langle\nabla(e),u\rangle),v\rangle + \langle\nabla(\langle\nabla(e),v\rangle),u\rangle = 0 
\end{array}$$
for any $u,v\in \bella$ and $ e\in \belE$, that is if and only if the curvature of $\nabla$ satisfies 
$$
F_\nabla = Q \id_\belE\ .
$$

So we have:
\begin{proposition}
Let $\bella$ be a holomorphic Lie algebroid, $Q\in H^0(X,\Omega^2_\bella)$ a closed $2$-$\bella$-form, and $\belE$ a coherent sheaf on $X$.

Then giving a $\Lambda_{\bella,[(0,Q)]}$-module structure $\mu$ on $\belE$ is equivalent to giving a holomorphic $\bella$-connection $\nabla: \belE\rightarrow \belE \otimes \la\dual$ such that $F_\nabla = Q \id_\belE$.
\end{proposition}

By virtue of this proposition and Theorem \ref{thm:simpson_main}, there exists quasi-projective moduli schemes $M_{\bella,Q}(P)$ that are coarse moduli spaces for semistable pairs $(\belE,\nabla)$, where $\belE$ is a torsion free $\corO_X$-module with Hilbert polynomial $P$, $\nabla$ is a holomorphic $\bella$-connection on $\belE$ satisfying $F_\nabla = Q\cdot \id_\belE$, and "semistable" means that for any subsheaf $\belF \subseteq \belE$ with $\nabla(\belF) \subseteq \belF \otimes \Omega_\bella$ one has $p(\belF) \leq p(\belE)$.

Now, if $Q,Q'$ are two cohomologous closed $2$-$\bella$-forms, the algebras $\Lambda_{\bella, [(0,Q)]},\Lambda_{\bella, [(0,Q')]}$ are the same, so the moduli spaces $M_{\bella,Q}(P),M_{\bella,Q'}(P)$ are naturally isomorphic. Moreover, if $Q$ is not cohomologous to $0$, the moduli spaces $M_{\bella,Q}(P)$ are empty for any polynomial $P$: on one side, by Theorem \ref{cor:vanishing_characteristic_ring}, if $\belE$ is a torsion free $\corO_X$-module and $\nabla$ is a holomorphic $\bella$-connection on it, then the cohomology class of the trace of $F_\nabla$ is zero, while on the other side $\trace(Q \cdot \id_\belE) = rnk(\belE) \cdot Q$.

Summing up, we have:
\begin{corollary} \label{cor:moduli}
Let $\bella$ be a holomorphic Lie algebroid over a smooth projective variety $X$, $Q\in H^0(X, \Omega^2_\bella)_{\text{closed}}$ and $P$ a numerical polynomial.

Then for $Q$ not cohomologous to $0$ the moduli spaces $M_{\bella,Q}(P)$ are empty, while for $Q$ cohomologous to $0$ the moduli spaces $M_{\bella,Q}(P)$ parametrize semistable flat holomorphic $\bella$-connections with Hilbert polynomial $P$.
\end{corollary}

\spazio
For a general $\Sigma\in F^1H^2(\bella,\C)$, we do not have a global splitting of the $1$st-order sequence. But, if $(Q_\alpha,\phi_{\alpha\beta})$ is a representative of $\Sigma$, we can choose local splittings $\zeta_\alpha$ over $U_\alpha$ such that $\zeta_\beta - \zeta_\alpha = \phi_{\alpha\beta}$ on the overlaps, and repeat the previous argument to find that a $\Lambda_{\bella,[(Q_\alpha,\phi_{\alpha\beta})]}$-module structure on a coherent sheaf $\belE$ is equivalent to a bunch of holomorphic $\bella$-connections $\nabla_{\alpha}$ on $\belE_{|U_\alpha}$ such that
\begin{itemize}
\item $F_{\nabla_{\alpha}} = Q_\alpha \id_{\belE|U_\alpha}$,
\item $\nabla_\beta - \nabla_\alpha = \id_{\belE} \otimes \phi_{\alpha\beta}$ over the double intersections $U_{\alpha\beta}$.
\end{itemize}

\subsection{Examples}
If the Lie algebroid is the canonical one $(\belT_X,\id,[\cdot,\cdot])$, then $\belT_X$-connections on a sheaf $\belE$ are just usual connections. So $M_{\belT_X}(P)$ is the moduli space of semistable flat connections with Hilbert polynomial $P$, usually denoted by $M_{DR}(P)$.

\spazio
If $(\belK,0,\{,\})$ is a holomorphic bundle of Lie algebras, then a $\belK$-connection $\nabla$ on a sheaf $\belE$ is an $\corO_X$-linear map $\nabla: \belE\rightarrow \belE\otimes \belK\dual$. Indeed, since the anchor is zero, so is the restriction of $\dc_\belK$ to functions. So $\nabla$ may be seen as a section of $\End(\belE)\otimes \bella\dual$. 
Remark that, since the anchor is $0$, the bracket is $\corO_X$-bilinear, hence we can see it as a section $\Theta \in H^0(\bigwedge^2 \bella\dual \otimes \bella)$. 

The curvature of a $\belK$-connection in this case is given by $F_\nabla = \nabla\wedge\nabla + \langle\Theta, \nabla\rangle \in H^0(\End(\belE) \otimes \bigwedge^2 \belK\dual)$.

A particular case of this is when $\belK=\belT_X$ equipped with the trivial Lie algebroid bundle structure. In this case a flat $(\belT_X,0,0)$-connection on $\belE$ is an $\corO_X$-linear map $\phi: \belE \rightarrow \belE\otimes \Omega_X$ satisfying $\phi\wedge\phi =0$, i.e. it is a \emph{Higgs field} on $\belE$.

Another interesting case is when $\belK=\Omega_X$ with the trivial structure: similarly to the above, a flat $(\Omega_X,0,0)$-connection on $\belE$ is a $\corO_X$-linear map $\phi: \belE\rightarrow \belE\otimes \belT_X$ satisfying $\phi\wedge\phi =0$, that is, a \emph{co-Higgs field}, recently introduced by Hitchin in \cite{hitchin}. This is a particular case of the construction described in the next section.

\subsection{Holomorphic Poisson structures and generalized complex geometry}
Let $X$ be a smooth projective variety and $\Pi \in H^0(X, \bigwedge^2 \belT_X)$ a Poisson bivector. As is Section \ref{sec:example_1}, it defines a holomorphic Lie algebroid structure on $\Omega_X$ that we shall denote by $(\Omega_X)_\Pi$.

According to \cite{gualtieri} and \cite{xu}, this Lie algebroid defines a generalized complex structure on $X$. It can be described as follows. Recall that a generalized complex structure on a smooth manifold $M$ is defined by the ($+i$)-eigenbundle $L$ of an endomorphism $\bJ$ of $(T_M\oplus T^*_M) \otimes \C$ satisfying $\bJ^2 = -1$ and $\bJ^* = - \bJ$. Since $L$ is an isotropic subbundle of the Courant algebroid $T_M \oplus T^*_M$, the resriction of the Courant bracket of $T_M\oplus T^*_M$ to $L$ defines a real Lie algebroid structure on it. 

When $X$ is a holomorphic Poisson manifold with holomorphic Poisson bivector $\Pi$, we can define the following endomorphism of $T_X\oplus T^*_X$:
$$
\bJ_{4\Pi} = \left( \begin{array}{cc}  -J & 4 \sharp_I \\
 0 &  J^*\end{array} \right)
$$
where $J$ is the almost complex structure on $X$ and $\sharp_I$ is the morphism $T^*_X \rightarrow T_X$ associated to the bivector $\Pi_I$, where $\Pi = \Pi_R + i \Pi_I$ for $\Pi_R,\Pi_I \in \Gamma (\bigwedge^2 T_X)$. It is easy to see that $\bJ_{4\Pi}$ defines a generalized complex structure on $X$, that we shall call $L_{4\Pi}$. 

Remark that the elements of $L_{4\Pi}$ are of the form $(V + i 4\sharp_I \xi , \xi)$ with $V\in T^{0,1}_X$ and $\xi \in T^{* 1,0}_X$, which gives an isomorphism $L_{4\Pi} \iso T^{0,1}_X \oplus T^{*1,0}_X$. Moreover we have $L_{4\Pi}^* \iso T^{*0,1}_X \oplus T^{1,0}_X$, and the Lie algebroid differential on functions is $d_{L_{4\Pi}} f = \bar{\partial} f + \sharp_I(\partial f)$.

The following theorem is proved in \cite{xu}:
\begin{theorem}
The Lie algebroid $L_{4\Pi}$ is isomorphic to $(\Omega_X)_\Pi \bowtie T^{0,1}_X$.
\end{theorem}

If $E$ is a vector bundle on $M$, we shall call a \emph{$L$-generalized holomorphic structure} on $E$ a representation of $L$ on $E$.

By Lemma \ref{lem:connections}, a $L_{4\Pi}$-generalized complex structure on a vector bundle $E$ is equivalent to a holomorphic structure $\belE$ on $E$ and a flat holomorphic $(\Omega_X)_\Pi$-connection on $\belE$. In particular, by the previous construction, we obtain moduli spaces for such objects, that is:
\begin{corollary}
Let $X$ be a smooth complex projective variety, and $\Pi\in H^0(X,\bigwedge^2 \belT_X)$ an algebraic Poisson bivector on $X$ inducing a generalized complex structure $L_{4\Pi}$ on $X$.

Then for any numerical polynomial $P$, there exists a quasi-projective scheme $M_\Pi(P)$ parametrizing semistable $L_{4\Pi}$-generalized holomorphic vector bundles with Hilbert polynomial $P$.
\end{corollary}

\end{document}